\documentclass{amsart}
\usepackage{graphicx,amssymb,amsthm,amsmath}

\newcommand{\sgn}{\operatorname{sgn}}

\theoremstyle{plain}
\newtheorem{theorem}{Theorem}[section]
\newtheorem{lemma}{Lemma}[section]
\newtheorem{corollary}{Corollary}[section]
\theoremstyle{definition}
\newtheorem{remark}{Remark}[section]
\newtheorem{definition}{Definition}[section]

\title[Approximating mixed H\"older functions]
{Approximating mixed H\"older functions \\ using random samples}

\author[N.~F.~Marshall]{Nicholas F. Marshall}
\address{Program in Applied Mathematics, Department of Mathematics, Yale University, New Haven, CT 06511,
USA} 
\email{nicholas.marshall@yale.edu}


\begin{document}

\begin{abstract}
Suppose $f : [0,1]^2 \rightarrow \mathbb{R}$ is a $(c,\alpha)$-mixed H\"older
function that we sample at $l$ points $X_1,\ldots,X_l$ chosen uniformly at
random from the unit square. Let the location of these points and the function
values $f(X_1),\ldots,f(X_l)$ be given. 
If $l \ge c_1 n \log^2 n$, then we can
compute an approximation $\tilde{f}$ such that
$$
\|f - \tilde{f} \|_{L^2} = \mathcal{O}(n^{-\alpha} \log^{3/2} n),
$$
with probability at least  $1 - n^{2 -c_1}$, where the implicit constant only
depends on the constants $c > 0$ and $c_1 > 0$.
\end{abstract}

\maketitle

\section{Introduction}
\subsection{Introduction} \label{intro}
A function $f :[0,1]^2 \rightarrow \mathbb{R}$ is
$(c,\alpha)$-mixed H\"older if
$$
|f(x^\prime,y) - f(x,y)| \le c |x^\prime - x|^\alpha, \qquad 
|f(x,y^\prime) - f(x,y)| \le c |y^\prime - y|^\alpha,
$$
and
$$
|f(x^\prime,y^\prime) - f(x,y^\prime) - f(x^\prime,y) + f(x,y)| \le c
(|x^\prime-x| |y^\prime -y|)^\alpha,
$$
for all $x,x^\prime,y,y^\prime \in [0,1]$. For example, if $f : [0,1]^2
\rightarrow \mathbb{R}$ satisfies 
$$
\left|\frac{\partial f}{\partial x} \right| \le c, \quad \left|\frac{\partial
f}{\partial y} \right| \le c, \quad \text{and} \quad \left| \frac{\partial
f}{\partial x \partial y} \right| \le c \quad \text{on $[0,1]^2$},
$$
then by the mean value theorem $f$ is $(c,1)$-mixed H\"older. In 1963, Smolyak
\cite{Smolyak1963} discovered a surprising approximation result for mixed
H\"older functions.
\begin{lemma}[Smolyak] \label{smolyak}
Suppose that $f:[0,1]^2 \rightarrow \mathbb{R}$ is $(c,\alpha)$-mixed H\"older.
Then
$$
f(x,y) = \sum_{k=0}^m f(x_k,y_{m-k}) - \sum_{k=1}^{m} f(x_{k-1},y_{m-k}) +
\mathcal{O} \left(m 2^{-\alpha m} \right),
$$
where $x_k$ is the center of the dyadic interval of length $2^{-k}$ that
contains $x$, and $y_j$ is the center of the dyadic interval of length
$2^{-j}$ that contains $y$.
\end{lemma}

Observe that the point $(x_k,y_j)$ is the center of a dyadic rectangle
of width $2^{-k}$ and height $2^{-j}$; thus, Lemma
\ref{smolyak} is a statement about approximating mixed H\"older functions by
linear combinations of function values at the center of dyadic rectangles of
area $2^{-m}$ and $2^{-m+1}$.  

We remark that Smolyak \cite{Smolyak1963}
actually presented a general $d$-dimensional version of Lemma \ref{smolyak}, and
that the ideas of Smolyak were expanded upon by Str\"omberg
\cite{Stromberg1998}, and have been developed into a computational tool called
sparse grids, see \cite{BungartzGriebel2004}.  The proof of Lemma \ref{smolyak}
involves a telescoping series argument and is included below; throughout, we use
the notation $f \lesssim g$ when $f \le C g$ for some constant $C > 0$.

\begin{proof}[Proof of Lemma \ref{smolyak}]
Fix $(x,y) \in [0,1]^2$. For notational brevity set $f_k^j :=
f(x_k,y_j)$. First, we approximate $f(x,y)$ by the center $f_m^m$ 
of a $2^{-m}$ by $2^{-m}$ square. Clearly,
$$
|f(x,y) - f_m^m| \lesssim 2^{-\alpha m}.
$$
Expanding $f_{m}^m$ in successive telescoping series in
$\{x_k\}_{k=1}^m$ and $\{y_j\}_{j=1}^m$ gives 
$$
f_m^m = \sum_{j=1}^m \sum_{k=1}^m \left( f_k^j  - f_{k-1}^j - f_k^{j-1} +
f_{k-1}^{j-1} \right) + \sum_{l=1}^m \left( f_l^0 - f_{l-1}^0
+ f_0^l - f_0^{l-1} \right) + f_0^0.
$$
Since $f$ is $(c,\alpha)$-mixed H\"older, it follows that the terms of
the double sum satisfy
$$
\left| f_k^j  - f_{k-1}^j - f_k^{j-1} + f_{k-1}^{j-1} \right| \lesssim 2^{-\alpha(j +
k)}.
$$
Thus, we can bound the sum of terms in the double sum such that $j + k
> m$ by
$$
\sum_{j=1}^m \sum_{k=m-j+1}^m 
\left| f_k^j  - f_{k-1}^j - f_k^{j-1} + f_{k-1}^{j-1} \right|  \lesssim
\sum_{j=1}^m \sum_{k=m-j+1}^m 2^{-\alpha(j+k)} \lesssim m 2^{-\alpha m}.
$$
Removing these terms from the double sum and collapsing the
telescoping series leaves only terms $f_j^k$ such that $j+k \in \{m, m-1\}$; in
particular, we conclude that
$$
\left| f_{m}^{m} - \left( \sum_{l=0}^m f_l^{m-l} - \sum_{l=1}^{m} f_{l-1}^{m-l}
\right) \right| \lesssim m 2^{-\alpha m},
$$
which completes the proof.
\end{proof}

\begin{remark}
The proof began by approximating $f(x,y)$ to error $\mathcal{O}(2^{-\alpha m})$
by the function value at the center of the dyadic square with side length
$2^{-m}$ which contains $(x,y)$.  However, it would require $2^{2 m}$ function
values to approximate $f$ at every point in the unit square using this method.
In contrast, the telescoping argument in the proof of Lemma \ref{smolyak}
achieves an approximation error of $\mathcal{O}(m 2^{-\alpha m})$ while
only using function values at the center of dyadic rectangles of area $2^{-m}$
and $2^{-m+1}$; the total number of such rectangles is $(m+1) 2^m + m 2^{m-1}$.
\end{remark}

\subsection{Main result}
Informally speaking, Lemma \ref{smolyak} says that if
we are given a specific set of $\sim n \ln n$ samples of a $(c,\alpha)$-mixed
H\"older function $f : [0,1]^2 \rightarrow \mathbb{R}$, then we are able to
compute an approximation $\tilde{f}$ of $f$ such that
$$
\|f - \tilde{f}\|_{L^\infty} = \mathcal{O}(n^{-\alpha} \log n),
\quad \text{and} \quad
\|f - \tilde{f}\|_{L^2} = \mathcal{O}(n^{-\alpha} \log n),
$$
where the $L^2$-norm estimate follows directly from the $L^\infty$-norm
estimate.  Our main result relaxes the
sampling requirement to $\sim n \log^2 n$ random samples and achieves the same
$L^2$-norm error estimate up to $\log$ factors.

\begin{theorem} \label{thm1}
Suppose $f : [0,1]^2 \rightarrow \mathbb{R}$ is a $(c,\alpha)$-mixed H\"older
function that we sample at $l$ points $X_1,\ldots,X_l$ chosen uniformly at
random from the unit square. Let the location of these points and the function
values $f(X_1),\ldots,f(X_l)$ be given. 
If $l \ge c_1 n \log^2 n$, then we can
compute an approximation $\tilde{f}$ such that
$$
\|f - \tilde{f} \|_{L^2} = \mathcal{O}(n^{-\alpha} \log^{3/2} n),
$$
with probability at least  $1 - n^{2-c_1}$, where the implicit constant only
depends on the constants $c > 0$ and $c_1 > 0$.
\end{theorem}

When $\alpha > 1/2$ the theorem implies that we can integrate mixed
H\"older functions  on the unit square with an error rate that is better than
the Monte Carlo rate of $\mathcal{O}(n^{-1/2})$ with high probability. 

\begin{corollary} \label{cor1}
Under the assumptions of Theorem \ref{thm1}, if $l \ge c_1 n \log^2 n$, then we
can compute an approximation $I$ of the integral of $f$ on $[0,1]^2$ such that
$$
\int_{[0,1]^2} f(x) dx = I + \mathcal{O}(n^{-\alpha} \log^{3/2} n),
$$
with probability at least  $1 - n^{2 -c_1}$. 
\end{corollary}

The proof of this corollary follows immediately from the $L^2$-norm estimate
from Theorem \ref{thm1} and the Cauchy Schwarz inequality. 

\begin{remark} \label{rmk1}
The computational cost of computing $\tilde{f}$ is $\mathcal{O}(n
\log^3 n)$ operations of pre-computation, and then $\mathcal{O}(\log n)$
operations for each point evaluation. Furthermore, after pre-computation we can
compute the integral of $\tilde{f}$ on the unit square in $\mathcal{O}(n)$
operations.  The construction of $\tilde{f}$ is described in \S \ref{proofmain}.
\end{remark}

\begin{remark} \label{spin}
An advantage of using random samples and Theorem \ref{thm1} 
to approximate a mixed H\"older function over using samples 
at the center of dyadic rectangles and Lemma \ref{smolyak} is the ability
to perform spin cycling. For simplicity of exposition, assume that $f :
[0,1]^2 \rightarrow \mathbb{R}$ is a mixed H\"older function on the torus. Let
$X_1,\ldots,X_l$ be chosen uniformly at random from $[0,1]^2$, and let the
function values $f(X_1),\ldots,f(X_l)$ be given. By Theorem \ref{thm1} we can
compute an approximation $\tilde{f}$ of the function $f$; however, as described
in \S \ref{proofmain} the computation of $\tilde{f}$ is dependent on the dyadic
decomposition of $[0,1]^2$, and this dependence will create artifacts. We call
the following method of removing these artifacts spin cycling.

Let $\zeta \in [0,1]^2$ be given, and define $f_\zeta(x) = f(x - \zeta)$ where
addition is performed on the torus. By considering the function values
$f(X_1),\ldots,f(X_l)$ as values of the function $f_\zeta$ at the uniformly
random sample of points $X_1+\zeta,\ldots,X_l+\zeta$, we can use Theorem
\ref{thm1} to compute an approximation $\tilde{f}_\zeta$ of the function
$f_\zeta$. It follows that $\tilde{f}_\zeta(x + \zeta)$ is an approximation of
$f$ with the same accuracy guarantees as $\tilde{f}$.  However, the shift
$\zeta$ has changed the relation of the function values to the dyadic
decomposition of $[0,1]^2$, and thus has changed the resulting artifacts. In
general, we can consider a sequence of shifts $\zeta_1,\ldots,\zeta_q \in
[0,1]^2$ and define
$$
\bar{f}(x) = \frac{1}{q} \sum_{k=1}^q \tilde{f}_{\zeta_k}(x + \zeta_k)
\quad \text{for} \quad x \in [0,1]^2,
$$
where $\tilde{f}_{\zeta_j}$ is the approximation of the function $f_{\zeta_j}$
computed via Theorem \ref{thm1} using the shift operation described
above. We say that $\bar{f}$ is an approximation via Theorem \ref{thm1} with
$q$ spin cycles. In \S \ref{example} we provide empirical evidence that spin
cycling removes artifacts. We note that when $l \ge c_1 n \log^2 n$ and $c_1 > 2
+ \log(q)/\log(n)$, it follows that the accuracy claims of Theorem \ref{thm1} hold for 
all function $\tilde{f}_{\zeta_j}$ for $j = 1,\ldots,q$ with high probability.
The assumption that $f$ is mixed H\"older on the torus can be relaxed by
handling the boundaries appropriately. We emphasize that spin cycling is not
possible when using a fixed sample of points at the center of dyadic rectangles
and Lemma \ref{smolyak} as any shift moves the points away from the center of
dyadic rectangles, which is prohibitive for using Lemma \ref{smolyak}.
\end{remark}

\section{Preliminaries}
\subsection{Notation}
Let $\mathcal{D}$ denote the set of dyadic intervals in $[0,1]$; more precisely,
$$
\mathcal{D} := \left\{ \left[ (j-1)2^{-k}, j 2^{-k} \right) \subset \mathbb{R} :
k \in \mathbb{Z}_{\ge 0} \wedge j \in \{1,\ldots,2^k\} \right\}.
$$
We say that $R = I \times J$ is a dyadic rectangle in the unit square
if $I,J \in \mathcal{D}$. The number of dyadic rectangles in the unit square of
area $2^{-m}$ is
$$
(m+1) 2^m = \# \left\{ R = I \times J : |R| = 2^{-m} \wedge I,J \in \mathcal{D}
\right\}.
$$
In particular, for each $k = 0,\ldots,m$ there are $2^m$ distinct dyadic
rectangles of width $2^{-k}$ and height $2^{m-k}$, which are disjoint and cover
the unit square. We illustrate the dyadic rectangles in the unit square of area
at least $2^{-3}$ in Figure \ref{fig1}.
\begin{figure}[h!]
\includegraphics[width=.4\textwidth]{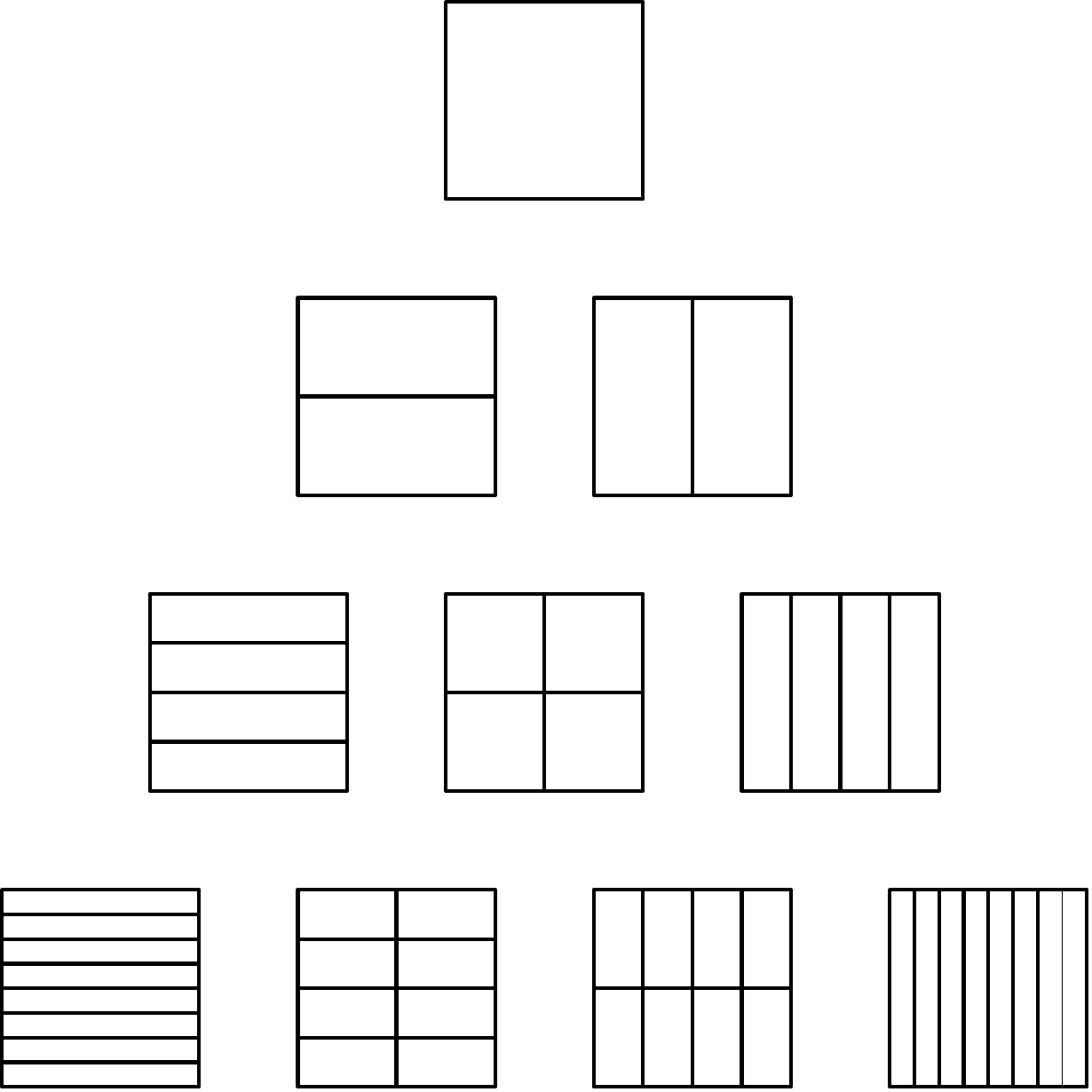}
\caption{The dyadic rectangles of area at least $2^{-3}$ in the unit square.}
\label{fig1}
\end{figure}

Recall that Lemma \ref{smolyak} approximates the value $f(x)$ of a
mixed H\"older function by a linear combination of the function values at the
centers of dyadic rectangles of area $2^{-m}$ and $2^{-m+1}$ that contain the
point $x$.  Thus, with respect to the illustration in Figure \ref{fig1}, the
approximation formula of Lemma \ref{smolyak} consists of adding the function
values at the center of the dyadic rectangles in the lowest row which contain
$x$, and subtracting the function values at the center of the dyadic rectangles
in the second lowest row which contain $x$.

\subsection{Randomized Kaczmarz} In addition to properties of dyadic rectangles,
we will use a result of Strohmer and Vershynin \cite{StrohmerVershynin2009}
regarding the convergence of a randomized Kaczmarz algorithm. Specifically,
Strohmer and Vershynin show that a specific randomized Kaczmarz algorithm
converges exponentially fast at a rate that only depends on how well the matrix
is conditioned. The following lemma is a special case of their result, which
will be sufficient for our purposes.

\begin{lemma}[Strohmer, Vershynin] \label{randomkaczmarz}
Let $A$ be an $N \times n$ matrix where $N \ge n$ whose rows are of equal
magnitude, and let $A w = b$ be a consistent linear system of equations. Suppose
that $l$ indices $I_1,\ldots,I_l$ are chosen uniformly at random from
$\{1,\ldots,N\}$.  Let an initial guess at the solution $v_0$ be given. For $k =
1,\ldots,l$ define
$$
v_k := v_{k-1} +  \frac{b_{I_k} -\langle a_{I_k}, v_{k-1} \rangle
}{\|a_{I_k}\|_{\ell^2}^2}
a_{I_k}, 
$$
where $a_j$ denotes the $j$-th row of $A$, and $b_j$ denotes the $j$-th entry of
$b$. Then
$$
\mathbb{E} \|v_k-w\|^2_{\ell^2} \le (1 - \kappa^{-2})^k \|v_0 -
w\|_{\ell^2}^2,
$$
for $k = 1,\ldots,l$, where $\kappa^2 := \sum_{j=1}^n \sigma_j^2/\sigma_n^2$
and $\sigma_1,\ldots,\sigma_n$ are the singular values of $A$.
\end{lemma}

The rate of convergence of the algorithm is determined by the constant
$\kappa$, which only depends on the singular values of the matrix $A$. This
constant $\kappa$ can be viewed as a type of condition number for the matrix
$A$, and can be equivalently defined as the Frobenius norm of $A$ multiplied by
the operator norm of the left inverse of $A$.  We remark that the convergence
of the randomized Kaczmarz algorithm for inconsistent linear systems $A w
\approx b + \varepsilon$ is analyzed by Needell \cite{Needell2010}.  In the
proof of the main result we use Lemma \ref{randomkaczmarz} in combination with a
modified version of the error analysis in \cite{Needell2010}.

\subsection{Organization}
The remainder of the paper consists of the proof of Theorem \ref{thm1} in \S
\ref{proofmain} followed by discussion in \S \ref{discussion}. The proof of
Theorem \ref{thm1} is organized as follows.  In \S \ref{embed}, we define an
embedding of the points in the unit square into a larger finite dimensional
vector space.  In \S \ref{martingale}, we show that inner products of vectors
with the defined embedding coordinates have a martingale interpretation. In \S
\ref{functional}, we show that mixed H\"older functions can be approximated by
linear functionals in the embedding coordinates.  In \S \ref{random}, we show
that the randomized Kaczmarz algorithm can be used to solve a specifically
constructed system. In \S \ref{complete}, we use the developed tools to complete
the proof of Theorem \ref{thm1}. Finally, in \S \ref{proofrmk} we prove the
computational cost claims of Remark \ref{rmk1}.

\section{Proof of Theorem \ref{thm1}} 
\label{proofmain}
\subsection{Embedding points} \label{embed}
Recall, that there are $m 2^{m-1}$ dyadic rectangles of area $2^{-m+1}$ in the
unit square $[0,1]^2$. Let 
$$
T_1,\ldots,T_{m 2^{m-1}}
$$
be an enumeration of these rectangles such that the rectangles 
$$
T_{k 2^{m-1}+1},\ldots,T_{(k+1) 2^{m-1}}
$$
have width $2^{-k}$ and height $2^{k-m+1}$ for $k = 0,\ldots,m-1$. Let $T^+_j$
and $T^-_j$ denote the left and right halves of $T_j$, respectively.
Furthermore, let
$$
R_1,\ldots,R_{2^m}
$$
be an enumeration of the dyadic rectangles of width $1$ and height $2^{-m}$. 

\begin{definition} \label{def1}
We define an embedding $\Psi : [0,1]^2 \rightarrow \mathbb{R}^{(m+2) 2^{m-1}}$
entry-wise by
$$
\Psi_j(x) = \chi_{R_j}(x) \quad \text{for} \quad j = 1,\ldots,2^{m},
$$
and
$$
\Psi_{2^m +j}(x) = \frac{1}{\sqrt{2}}\left( \chi_{T_j^+}(x) - \chi_{T_j^-}(x)
\right),
\quad \text{for} \quad j = 1,\ldots,m 2^{m-1},
$$
where $\chi_{R}$ denotes the indicator function for the rectangle $R$.
\end{definition}

Fix $x \in [0,1]^2$, and let $\beta_0$ be the index of the dyadic rectangle
$R_{\beta_0}$ of width $1$ and height $2^{-m}$ that contains $x$. Then, for
$k=1,\ldots,m$, let $\beta_k - 2^m$ be the index of the dyadic rectangle
$T_{\beta_k - 2^m}$ of width $2^{-k+1}$ and height $2^{k-m}$ that contains $x$.
Set $\xi_0 := \Psi_{\beta_0}(x) = 1$, and 
$$
\xi_k := \Psi_{\beta_k}(x) = \frac{1}{\sqrt{2}} \left( \chi_{T^+_{\beta_k -
2^m}}(x) - \chi_{T^-_{\beta_k - 2^m}}(x) \right), \quad \text{for} \quad k =
1,\ldots,m,
$$
such that $\xi_k$ is $+1/\sqrt{2}$ or $-1/\sqrt{2}$ depending on if $x$ is
contained in the left or right half of $T_{\beta_k-2^m}$, respectively. Then, if
$v \in \mathbb{R}^{(m+2)2^{m-1}}$ we have
$$
\langle \Psi(x), v \rangle = \sum_{k=0}^m \xi_k v_{\beta_k},
$$
where $\langle \cdot , \cdot \rangle$ denotes the $(m+2)2^{m-1}$-dimensional
Euclidean inner product. In the following section we show that partial sums of
this inner product can be interpreted as martingales.

\subsection{Martingale interpretation} \label{martingale}
Suppose that $x \in [0,1]^2$ is chosen uniformly at random, and let the
indices $\beta_0,\ldots,\beta_m$ and the scalars $\xi_0,\ldots,\xi_m$ be defined
as above.  Let $v \in \mathbb{R}^{(m+2) 2^{m-1}}$ be a fixed unit vector. We
define the partial sum $Y_r$ by
$$
Y_r = \sum_{k=0}^{r} \xi_k v_{\beta_k}, \quad \text{for} \quad
r = 0,\ldots,m.
$$
We assert that $\{Y_r\}_{k=0}^m$ is a martingale with respect to
$\{\beta_0,\xi_1,\ldots,\xi_m\}$, that is, 
$$
\mathbb{E} \left( Y_{k+1} \big| \beta_0,\xi_1,\ldots,\xi_k \right) = Y_k,
\quad \text{for} \quad k = 0,\ldots,m-1.
$$
Indeed, this martingale property can be seen by interpreting the partial sums
from a geometric perspective. Recall that $\beta_0$ determines the dyadic
rectangle $R_{\beta_0}$ of width $1$ and height $2^{-m}$ that contains $x$.
Therefore, $\beta_0$ determines the dyadic rectangle $T_{\beta_1 - 2^m}$ of
width $1$ and height $2^{-m+1}$ that contains $x$. However, $\beta_0$ provides
no information about $\xi_1 = \pm 1/\sqrt{2}$, which is positive or negative
depending on if the point $x$ is in the left or right side of $T_{\beta_1-2^m}$,
respectively. It follows that
$$
\mathbb{E} \left( Y_1 \big| \beta_0 \right) = \frac{1}{2} \left( 
v_{\beta_0} + \frac{1}{\sqrt{2}} v_{\beta_1}  \right) + \frac{1}{2} \left(
v_{\beta_0} - \frac{1}{\sqrt{2}} v_{\beta_1} \right) = Y_0.
$$
More generally, $\beta_k$ and $\xi_k$ determine $\beta_{k+1}$
since together $\beta_k$ and $\xi_k$ determine the dyadic rectangle $T_{\beta_k
- 2^m}^{\sgn \xi_k}$ of width $2^{-k}$ and height $2^{m-k}$, which contains $x$.
This, in turn, determines the rectangle $T_{\beta_{k+1}-2^m}$ of width $2^{-k}$
and height $2^{k-m+1}$ which contains $x$, but provides no information about
which side (left or right) of this rectangle the point $x$ is contained in, that
is to say, no information about $\xi_{k+1}$.  Hence
$$
\mathbb{E} \left( Y_{k+1} \big| \beta_0,\xi_1,\ldots,\xi_k \right) = 
\frac{1}{2} \left( Y_k + \frac{1}{\sqrt{2}} v_{\beta_{k+1}} \right)
+
\frac{1}{2} \left( Y_k - \frac{1}{\sqrt{2}} v_{\beta_{k+1}} \right)
= Y_k.
$$
This martingale property of the partial sums has several useful consequences.

\begin{lemma} \label{mart}
Suppose that $X$ is chosen uniformly at random from the unit square, and set $Y
= \Psi(X)$. Let $v \in \mathbb{R}^{(m+2) 2^{m-1}}$ be a fixed vector of unit
length. Then,
$$
\mathbb{E} |\langle Y, v \rangle|^2 = 2^{-m}. 
$$
\end{lemma}
\begin{proof}
Let $\beta_0,\ldots,\beta_m$ and $\xi_0,\ldots,\xi_m$ be as defined above such
that
$$
\mathbb{E} |\langle Y, v \rangle|^2 = \mathbb{E} \left| \sum_{k=0}^m \xi_k
v_{\beta_k} \right|^2 
= \sum_{k_1,k_2=0}^m \mathbb{E}  \xi_{k_1} \xi_{k_2} v_{\beta_{k_1}}
v_{\beta_{k_2}}.
$$
If $k_1 > k_2$, then $\xi_{k_2}$ and $\beta_{k_2}$ are determined
by $\beta_0,\xi_1,\ldots,\xi_{{k_1}-1}$; we conclude that
$$
\mathbb{E}  \left( \xi_{k_1} \xi_{k_2} v_{\beta_{k_1}} v_{\beta_{k_2}} \right) =
\mathbb{E} \left( \xi_{k_2}  v_{\beta_{k_2}} \mathbb{E} \left( \xi_{k_1}
v_{\beta_{k_1}} \big| \beta_0,\xi_1,\ldots,\xi_{k_1-1} \right) \right) = 0,
$$
where the finally equality follows from the fact that the expected value of
$\xi_{k_1} v_{\beta_{k_1}}$ conditional on $\beta_0,\xi_1,\ldots,\xi_{k_1-1}$ is
zero by the above described martingale property.  An identical argument holds
for the case when $k_1 < k_2$ so it follows that
$$
\sum_{k_1,k_2=0}^m \mathbb{E}  \xi_{k_1} \xi_{k_2} v_{\beta_{k_1}}
v_{\beta_{k_2}} = \sum_{k=0}^m \mathbb{E} \xi_{k}^2 v_{\beta_k}^2 =
\mathbb{E} v_{\beta_0}^2 + \frac{1}{2} \sum_{k=1}^m \mathbb{E} v_{\beta_k}^2.
$$
We can compute this expectation explicitly by noting that the probability that
$x$ is contained a given dyadic rectangle is proportional to its area;
specifically, we have
$$
\mathbb{E} v_{\beta_0}^2 + \frac{1}{2} \sum_{k=1}^m \mathbb{E} v_{\beta_k}^2 =
\frac{1}{2^m} \sum_{j=1}^{2^m} v_j^2 + \frac{1}{2} \frac{1}{2^{m-1}}
\sum_{k=1}^m \sum_{j=1}^{2^{m-1}} v_{(k+1)2^{m-1} + j}^2 = 2^{-m},
$$
where the final equality follows from collecting terms and using the assumption
that $v$ is a unit vector in $\mathbb{R}^{(m+2)2^{m-1}}$.
\end{proof}

Since embedding $\Psi : [0,1]^2 \rightarrow \mathbb{R}^{(m+2)2^{m-1}}$ is
defined using indicator functions of dyadic rectangles of area $2^{-m}$ in
$[0,1]^2$, it follows that $\Psi$ is constant on $2^{-m}$ by $2^{-m}$ dyadic
squares since all points in such a square are contained in the same collection
of dyadic rectangles of area $2^{-m}$ in $[0,1]^2$. This observation leads the
following corollary of Lemma \ref{mart}.

\begin{corollary} \label{corsing}
Let $x_1,\ldots,x_{2^{2m}}$ be a sequence of points such that each $2^{-m}$ by
$2^{-m}$ dyadic square contains exactly one point.  Let $A$ be the $2^{2m}
\times (m+2) 2^{m-1}$ matrix whose $j$-th row is $\Psi(x_j)$. Then,
$$
\sigma_1 = \cdots = \sigma_{(m+2)2^{m-1}} = 2^{m/2},
$$
where $\sigma_1,\ldots,\sigma_{(m+2)2^{m-1}}$ are the singular values of $A$.
\end{corollary}

\begin{proof}
Let $v \in \mathbb{R}^{(m+2)2^{m-1}}$ be an arbitrary unit vector. We have
$$
\|A v\|_{\ell^2}^2 = \sum_{j=1}^{2^{2m}} |\langle \Psi(x_j), v \rangle|^2.
$$
However, since all points in each $2^{-m}$ by $2^{-m}$ dyadic square have the
same embedding, and since the measures of all such dyadic squares are equal, we
have
$$
\sum_{j=1}^{2^{2m}} |\langle \Psi(x_j), v \rangle|^2 = 2^{2m} \mathbb{E} |
\langle Y, v \rangle|^2,
$$
where $Y := \Psi(X)$ for a point $X$  chosen uniformly at random from the unit
square. By Lemma \ref{mart} we conclude that
$$
\|A v \|_{\ell^2}^2 =  2^{2m} \mathbb{E} | \langle Y, v \rangle|^2 = 2^m,
$$
and since $v$ was an arbitrary unit vector the proof is complete. 
\end{proof}

\subsection{Approximation by linear functionals} \label{functional}
So far we have constructed an embedding $\Psi : [0,1]^2 \rightarrow
\mathbb{R}^{(m+2)2^{m-1}}$, and we have shown that inner products of the form
$\langle \Psi(X), v \rangle$ are related to martingales. We have used this
relation to show that the collection of all possible embedding vectors form a
matrix whose singular values are all $2^{m/2}$.  Next, we show that a mixed
H\"older function can be approximated by a linear functional in the embedding
coordinates.

\begin{lemma} \label{lem2}
Let $f : [0,1]^2 \rightarrow \mathbb{R}$ be a $(c,\alpha)$-mixed H\"older
function. Then, there exists a vector $w \in \mathbb{R}^{(m+2)2^{m-1}}$ such
that
$$
f(x) = \langle \Psi(x), w \rangle + \mathcal{O}(m 2^{-\alpha m}),
\quad \text{for all} \quad x \in [0,1]^2,
$$
where the vector $w$ depends on $f$, but is independent of $x$ and is explicitly
defined below in Definition \ref{def2}.
\end{lemma}

We construct the vector $w$ using a scheme similar to the construction
of Haar wavelets.  Let $\mathcal{D}_k^j$ be the collection of dyadic rectangles
contained in the unit square of width $2^{-k}$ and area $2^{-j}$. For a given
dyadic rectangle $R$, we define $s_r(R)$ by
$$
s_r(R) := \sum_{R^\prime \in \mathcal{D}_{r}^{m} : |R \cap R^\prime| > 0}
f(c_{R^\prime})
-
\sum_{R^\prime \in \mathcal{D}_{r}^{m-1} : |R \cap R^\prime| > 0}
f(c_{R^\prime}),
$$
where $c_{R^\prime}$ is the center of $R^\prime$. Observe that the first sum in
the definition of $s_r(R)$ is over the dyadic rectangles of width $2^{-r}$ and
area $2^{-m}$ that intersect $R$, while the second sum is over the dyadic
rectangles of width $2^{-r}$ and area $2^{-m+1}$ that intersect $R$. 

\begin{definition} \label{def2}
We define the vector $w \in
\mathbb{R}^{(m+2)2^{m-1}}$ entry-wise by
$$
w_j = \sum_{r=0}^m 2^{-r} s_r \left(R_j \right) \quad \text{for} \quad
j = 1,\ldots,2^m,
$$
and
$$
w_{2^m+j} = \frac{2^{k_j}}{\sqrt{2}}\sum_{r= k_j}^m 
2^{-r}\left( s_r \left( T_{j}^+ \right) - s_r\left( T_{j}^- \right) \right),
\quad \text{for} \quad j = 1,\ldots,m 2^{m-1},
$$
where $k_j :=\lfloor j 2^{-k} \rfloor$ is such that $2^{-k_j}$ is the width of
the rectangle $T_{j}$. 
\end{definition}

\begin{proof}[Proof of Lemma \ref{lem2}]
Let $x$ be a fixed point in the unit square $[0,1]^2$. Recall that we can
express the inner product
$$
\langle \Psi(x), w \rangle = \sum_{k=0}^m \xi_k w_{\beta_k},
$$
where the scalars $\xi_0,\ldots,\xi_m$ and the indicies $\beta_0,\ldots,\beta_m$
are as defined above.  First, let us rewrite Lemma \ref{smolyak} using this
notation. We have
$$
f(x) = f\left( c_{R_{\beta_0}} \right) + \sum_{k=1}^m f \left( c_{T_{\beta_k -
2^m}^{\sgn \xi_k}} \right) - \sum_{k=1}^m f \left( c_{T_{\beta_k - 2^m}} \right)
+ \mathcal{O}(m 2^{-\alpha m}).
$$
Indeed, by definition $R_{\beta_0}$ is the dyadic rectangle of width $1$ and
height $2^{-m}$ that contains $x$, $T_{\beta_k-2^m}^{\sgn \xi_k}$ is the dyadic
rectangle of width $2^{-k}$ and height $2^{m-k}$ that contains $x$, and
$T_{\beta_k-2^m}$ is the dyadic rectangle of width $2^{-k+1}$ and height
$2^{m-k}$ that contains $x$.  Thus, to complete the proof it suffices to show
that
$$
\sum_{k=0}^m \xi_k w_{\beta_k} = 
f\left( c_{R_{\beta_0}} \right) + \sum_{k=1}^m f \left( c_{T_{\beta_k -
2^m}^{\sgn \xi_k}} \right) - \sum_{k=1}^m f \left( c_{T_{\beta_k - 2^m}}
\right).  
$$
Let us start by considering the terms $\xi_0 w_{\beta_0},\ldots,\xi_m
w_{\beta_m}$ of the summation expression for the inner
product $\langle \Psi(x), w \rangle$. By the defintion of $w$ we have
$$
\xi_0 w_{\beta_0} = s_0(R_{\beta_0}) + \sum_{r=1}^m 2^{-r+1} \frac{
s_r(T_{\beta_1}^+) + s_r(T_{\beta_1}^-)}{2},
$$
and
$$
\xi_k w_{\beta_k} = \sum_{r= k}^m
2^{-r+k} \sgn \xi_k \frac{s_r \left( T_{\beta_k-2^m}^+ \right) - s_r\left(
T_{\beta_k-2^m}^- \right)}{2}.
$$
for $k = 1,\ldots,m$. We assert that if we start summing at $r = k +1$ we have
$$
\sum_{r=k+1}^m 2^{-r +k} s_r \left( T_{\beta_k-2^m}^{\sgn \xi_k} \right) =
\sum_{r=k+1}^m 2^{-r +k + 1} \frac{ s_r \left( T_{\beta_{k}-2^m}^{+} \right)+
s_r \left( T_{\beta_{k}-2^m}^{-} \right)}{2}.
$$
Indeed, observe that $T_{\beta_k-2^m}^{\sgn \xi_k}$ is the dyadic rectangle of
width $2^{-k}$ and height $2^{m-k}$ that contains $x$. We have that
$$
T_{\beta_k - 2^m}^{\sgn \xi_k} \subset T_{\beta_{k+1} - 2^m}^{+} \cup
T_{\beta_{k+1}-2^m}^- = T_{\beta_{k+1} - 2^m},
$$
since $T_{\beta_{k+1} - 2^m}$ is the dyadic rectangle of width $2^{-k}$ and
height $2^{k-m+1}$ that contains $x$. However, when $r \ge k+1$ we are summing
of dyadic rectangles of height at least $2^{m-k+1}$, and any dyadic rectangle of
height at least $2^{k-m+1}$ that intersects $T_{\beta_{k+1} - 2^m}$ must also
intersect $T_{\beta_{k}-2^m}^{\sgn \xi_k}$ so we conclude the above equality.
By applying the identity iteratively as we add each term $\xi_k w_{\beta_k}$ we
conclude that
$$
\sum_{k=0}^m \xi_k w_{\beta_k} = s_0(R_{\beta_0}) + \sum_{k=1}^m 
s_r \left(T_{\beta_k - 2^m}^{\sgn \xi_k} \right).
$$
Next, we observe that
$$
s_0(R_{\beta_0}) = f \left(c_{R_{\beta_0}} \right) - f \left( c_{T_{\beta_1 -
2^m}} \right),
$$
and that for $k = 1,\ldots,m-1$.
$$
s_k \left(T_{\beta_k - 2^m}^{\sgn \xi_k} \right) = 
f \left(c_{T_{\beta_k - 2^m}^{\sgn \xi_k}} \right) 
- f \left(c_{T_{\beta_{k+1} - 2^m}} \right).
$$
However, observe that when $r = m$ we have
$$
s_m(R) := \sum_{R^\prime \in \mathcal{D}_{m}^{m} : |R \cap R^\prime| > 0}
f(c_{R^\prime})
-
\sum_{R^\prime \in \mathcal{D}_{m}^{m-1} : |R \cap R^\prime| > 0}
f(c_{R^\prime}),
$$
and there are no rectangles in the set $\mathcal{D}_m^{m-1}$, which
is the set of rectangles of width $2^{-m}$ and area $2^{-m+1}$ that are
contained in the unit square; indeed, such a rectangle would need to have height
$2$, which is prohibitive.  We conclude that
$$
s_m\left( T_{\beta_m - 2^m}^{\sgn \xi_m} \right) = f \left( c_{T_{\beta_m -
2^m}^{\sgn \xi_m}} \right),
$$
Recall that we have already shown that
$$
\sum_{k=0}^m \xi_k w_{\beta_k} = s_0(R_{\beta_0}) + \sum_{k=1}^m 
s_k \left(T_{\beta_k -2^m}^{\sgn \xi_k} \right);
$$
substituting in the derived expressions for $s_0(R_{\beta_0})$ and
$s_k(T_{\beta_k}^{\sgn \xi_k})$ gives
$$
\sum_{k=0}^m \xi_k w_{\beta_k} = f\left( c_{R_{\beta_0}} \right) + \sum_{k=1}^m
f \left( c_{T_{\beta_k - 2^m}^{\sgn \xi_k}} \right) - \sum_{k=1}^m f \left(
c_{T_{\beta_k - 2^m}} \right),
$$
which completes the proof.
\end{proof}

\subsection{Random projections} \label{random}
We have established that in the embedding coordinates $\Psi(x)$ of a point $x
\in [0,1]^2$ that Smolyak's Lemma can be rephrased as a result about
approximating mixed H\"older functions by linear functionals. Moreover, using
the martingale interpretation of inner products of vectors with $\Psi(x)$ we
were able to explicitly compute the singular values of the matrix of all
possible embedding vectors. In the following we combine these ideas using the
randomized Kaczmarz algorithm of Strohmer and Vershynin
\cite{StrohmerVershynin2009}.

Suppose that $x_1,\ldots,x_{2^{2m}}$ is a sequence of points that contains
exactly one point in each $2^{-m}$ by $2^{-m}$ dyadic square in $[0,1]^2$. Let
$A$ be the $2^{2m} \times (m+2)2^{m-1}$ dimensional matrix whose $j$-th row is
$\Psi(x_j)$. Since the embedding $\Psi(x)$ has $1$ entry of magnitude $1$ and
$m$ entries of magnitude $1/\sqrt{2}$, see Definition \ref{def1}, we have
$$
\|\Psi(x) \|_{\ell^2} = \sqrt{1+m/2},
$$
for all $x \in [0,1]^2$, and it follows that all of the rows of $A$ have equal
magnitude. Suppose that $f : [0,1]^2 \rightarrow \mathbb{R}$ is a
$(c,\alpha)$-mixed H\"older function.  By Lemma \ref{lem2}, there exists a
vector $w$ such that
$$
| f(x) - \langle \Psi(x), w \rangle | \lesssim m 2^{-\alpha m}.
$$
Define
$$
\bar{f}(x) := \langle \Psi(x), w \rangle,
$$
for all $x \in [0,1]^2$.  If $b$ is the $2^{2m}$-dimensional vector whose $j$-th
entry is $\bar{f}(x_j)$, then we have a consistent linear system of equations
$$
A w = b.
$$
By Corollary \ref{corsing} the condition number $\kappa^2$ of $A$ satisfies
$$
\kappa^2 := \sum_{j=1}^{(m+2)2^{m-1}} \sigma_j^2/\sigma_{(m+2)2^{m-1}}^2 =
(m+2) 2^{m-1},
$$
where $\sigma_1,\ldots,\sigma_{(m+2)2^{m-1}}$ are the singular values of $A$.
Observe that sampling points uniformly at random from $[0,1]^2$ and applying the
embedding $\Psi$ is equivalent to choosing rows uniformly at random from $A$.
Thus, the following result is a direct consequence of applying Lemma
\ref{randomkaczmarz} to the consistent linear system of equations $A w = b$ that
we constructed above.

\begin{lemma} \label{lemrand}
Suppose that  $l$ points $X_1,\ldots,X_l$ are sampled uniformly at
random from $[0,1]^2$. Given an initial vector $v_0 \in
\mathbb{R}^{(m+2)2^{m-1}}$, define
$$
v_k := v_{k-1} + \frac{\bar{f}(X_k) - \langle \Psi(X_k), v_{k-1} \rangle}{1+m/2}
\Psi(X_k).
$$
Then,
$$
\mathbb{E} \|v_k - w\|^2_{\ell^2} \le \left(1 - \frac{1}{(m+2) 2^{m-1}}
\right)^k \|v_0 -w \|^2_{\ell^2},
$$
for $k = 1,\ldots,l$.
\end{lemma}

Note that $v_k$ is defined using $\bar{f}(x)$ rather than $f(x)$ so that the
definition of $v_k$ corresponds to running the randomized Kaczmarz
algorithm on the consistent linear system $A w = b$.  When we complete the proof
of Theorem \ref{thm1} in the following section, we estimate the error caused by
replacing $\bar{f}(x)$ by $f(x) = \bar{f}(x) + \mathcal{O}(m 2^{-\alpha m})$. We
remark that the expected error for the randomized Kaczmarz algorithm for
inconsistent linear systems is analyzed by Needell \cite{Needell2010}. Since we
need an error estimate that holds with high probability we perform a modified
version of the error analysis of Needell.

\subsection{Proof of Theorem \ref{thm1}} \label{complete} In this section, we
combine the developed tools to complete the proof of Theorem \ref{thm1}.
\begin{proof}[Proof of Theorem \ref{thm1}]
Suppose that $f : [0,1]^2 \rightarrow \mathbb{R}$ is a $(c,\alpha)$-mixed
H\"older function that is sampled at $l$ points $X_1,\ldots,X_l$ chosen
uniformly at random from $[0,1]^2$. For some initial vector $v_0^* \in
\mathbb{R}^{(m+2)2^{m-1}}$, define
$$
v_k^* := v_{k-1}^* + \frac{f(X_k) - \langle \Psi(X_k), v_{k-1}^* \rangle}{1+m/2}
\Psi(X_k),
$$
for $k = 1,\ldots,l$. Recall that  by Lemma \ref{lem2} there exists a vector
$w$ such that 
$$
|f(x) - \langle \Psi(x), w \rangle | \lesssim m 2^{-\alpha m},
$$
and recall that $\bar{f}(x) := \langle \Psi(x), w \rangle$. Suppose that
$\varepsilon_k := f(X_k) - \bar{f}(X_k)$. We can write
$$
v_k^* = v_k + e_k,
$$
where $v_k$ is the vector defined in Lemma \ref{lemrand} by
$$
v_k := v_{k-1} + \frac{\bar{f}(X_k) - \langle \Psi(X_k), v_{k-1} \rangle}{1+m/2}
\Psi(X_k),
$$
and $e_k$ is an error term defined by
$$
e_k := e_{k-1} + \frac{\varepsilon_k - \langle \Psi(X_k), e_{k-1}
\rangle}{1+m/2} \Psi(X_k).
$$
By orthogonality we have
$$
\|e_k\|_{\ell^2}^2 = \left\|e_{k-1} - \frac{\langle \Psi(X_k),e_{k-1}
\rangle}{1+m/2} \Psi(X_k) \right\|^2_{\ell^2} +
\frac{\varepsilon_k^2}{(1+m/2)^2} \| \Psi(X_k)\|^2_{\ell^2}.
$$
It follows that
$$
\|e_k\|_{\ell^2}^2 \le \|e_{k-1}\|_{\ell^2}^2 + \frac{\varepsilon_k^2}{1+m/2}
\lesssim k m 2^{-2 \alpha m}.
$$
By the triangle inequality we have
$$
\|v^*_k - w\|_{\ell^2} \lesssim \|v_k - w\|_{\ell^2} + \sqrt{k m} 2^{-\alpha m}.
$$
Next, we estimate $\|v_k - w\|_{\ell^2}$. From Lemma \ref{lemrand} we have
$$
\mathbb{E} \|v_k - w\|^2_{\ell^2} \le \left(1 - \frac{1}{(m+2) 2^{m-1}}
\right)^k \|v_0 -w \|^2_{\ell^2}.
$$
Thus, if $l \ge c_1 \log(2^m) (m+2) 2^{m-1}$, then we have
$$
\mathbb{E} \|v_l - w\|^2_{\ell^2} \le 2^{-c_1 m}  \|v_0 -
w\|_{\ell^2}^2.
$$
By the possibility of considering the function $f - f(X_1)$ instead of $f$, we
may assume that $|f| \le 2c$ on $[0,1]^2$. It follows that $\|w\|_{\ell^\infty}
\le 3c$ when $m$ is large enough. Therefore, if we initialize $v_0$ as the zero
vector
we have
$$
\|v_0 - w\|_{\ell^2}^2 \le 9c^2 (m+2) 2^{m-1}.
$$
From this estimate and our above analysis it follows that
$$
\mathbb{E} \|v_l - w\|^2_{\ell^2} \le \frac{9 c^2}{2} (m+2) 2^{(1-c_1) m},
$$
when $l \ge c_1 \log(2^m)  (m+2) 2^{m-1}$.  Observe that
$$
l = c_1 \log(2^m) (m+2) 2^{m-1} \le c_1 \log^2(2^m) 2^m,
$$
when $m$ is sufficiently large.  Thus, if $l \ge c_1 \log^2(2^m) 2^m$, then by
Markov's inequality
$$
\mathbb{P}(\|v_l - w\|_{\ell^2}^2 \ge m^3 2^{(1-2 \alpha) m}) \le
\frac{\mathbb{E} \|v_l-w\|_{\ell^2}^2}{m^3 2^{(1-2\alpha)m}} \le 2^{(2-c_1) m},
$$
when $m$ is large enough in terms of $c$.  Recall that we previously showed that
$$
\|v_l^* - w\|_{\ell^2} \lesssim \|v_l - w\|_{\ell^2} + \sqrt{m l} 2^{-\alpha m}.
$$
If $l = \lceil c_1 \log(2^m)^2 2^m \rceil$, then by our estimate on
$\|v_l-w\|_{\ell^2}$ it follows that
$$
\|v_l^* - w \|_{\ell^2}  \lesssim m^{3/2} 2^{(1/2-\alpha)m},
$$
with probability at least $1 - 2^{(2-c_1)m}$. By Corollary \ref{corsing},
the operator norm of $A$ is $2^{m/2}$, so it follows that
$$
\| A v_l^* - A w\|_{\ell^2} \lesssim m^{3/2} 2^{(1-\alpha)m},
$$
with probability at least $1 - 2^{(2-c_1)m}$.  Thus, if we define the function
$\tilde{f} : [0,1]^2 \rightarrow \mathbb{R}$ by
$$
\tilde{f}(x) := \langle \Psi(x), v_l^* \rangle,
$$
then we have the estimate
$$
\sqrt{\int_{[0,1]^2} |\tilde{f}(x) - \bar{f}(x)|^2 dx} = \|\tilde{f} -
\bar{f}\|_{L^2} \lesssim 2^{-\alpha m} m^{3/2}, 
$$
with probability at least $1 - 2^{(2 - c_1)m}$.  Since $\|\bar{f} - f\|_{L^2}
\lesssim 2^{-\alpha m} m$ it follows that $\|\tilde{f} - f\|_{L^2} \lesssim
2^{-\alpha m} m^{3/2}$. Setting $n :=2^m$ completes the proof.
\end{proof}

\subsection{Proof of Remark \ref{rmk1}} \label{proofrmk}
It remains to verify the computational cost claims of Remark \ref{rmk1}.
\begin{proof}[Proof of Remark \ref{rmk1}]
Let $n = 2^m$.  The computation of $v_l^*$ described in Lemma \ref{lemrand}
consists of $\mathcal{O}(n \log^2 n)$ iterations of $\mathcal{O}(\log n)$
operations for a total of $\mathcal{O}(n \log^3 n)$ operations of
pre-computation. Then, since $\Psi(x)$ is supported on $\mathcal{O}(\log n)$
entries, the inner product $\langle \Psi(x), v_l^* \rangle$ requires
$\mathcal{O}(\log n)$ operations. Finally, approximating the integral
of $f$ amounts to approximating the function $f$ at each point, taking the
sum, and dividing by $n^2$:
$$
\left| \int_{[0,1]^2} f(x) - \frac{1}{n^2} \left\langle \sum_{j=1}^{n^2}
\Psi(x_j), w \right\rangle  \right| \lesssim n^{-\alpha} \log^{3/2} n.
$$
However, this naive approach would require $\mathcal{O}(n^2)$ operations.
Instead, we make the observation that
$$
\sum_{j=1}^{n^2} \Psi(x_j) = g,
$$
where $g$ is the vector whose first $n$ entries are equal to $1$ and which is
zero elsewhere. Indeed, after the first $n$ entries, each entry is $+1/\sqrt{2}$
and $-1/\sqrt{2}$ for an equal number of embedding vectors. It follows that 
$$
\left| \int_{[0,1]^2} f(x) - \frac{1}{n} \left\langle g, w \right\rangle
\right| \lesssim n^{-\alpha} \log^{3/2} n;
$$
the computation of the inner product $\langle g,w \rangle$ only requires
$\mathcal{O}(n)$ operations so the proof is complete.
\end{proof}

\section{Discussion} \label{discussion}

\subsection{Illustration} \label{example}
Suppose that $f : [0,1]^2 \rightarrow \mathbb{R}$ is
the function defined by
$$
f(x,y) = \sin(20 x^2 + 10 y) \sin(\pi x) \sin(\pi y), \quad \text{for} \quad
(x,y) \in [0,1]^2.
$$
The function $f$ is $(c,1)$-mixed H\"older for some $c > 0$ since the partial
derivatives $\partial f/\partial x$, $\partial f/\partial y$, and $\partial^2 f
/ (\partial x \partial y)$ are bounded in $[0,1]^2$. As a baseline, in Figure
\ref{fig2} we plot the function $f$, and the approximation of $f$ via the method
of Smolyak (Lemma \ref{smolyak}) with $m = 7$ such that $n := 2^m = 128$.

\begin{figure}[h!]
\begin{tabular}{cc}
\includegraphics[width=.45\textwidth]{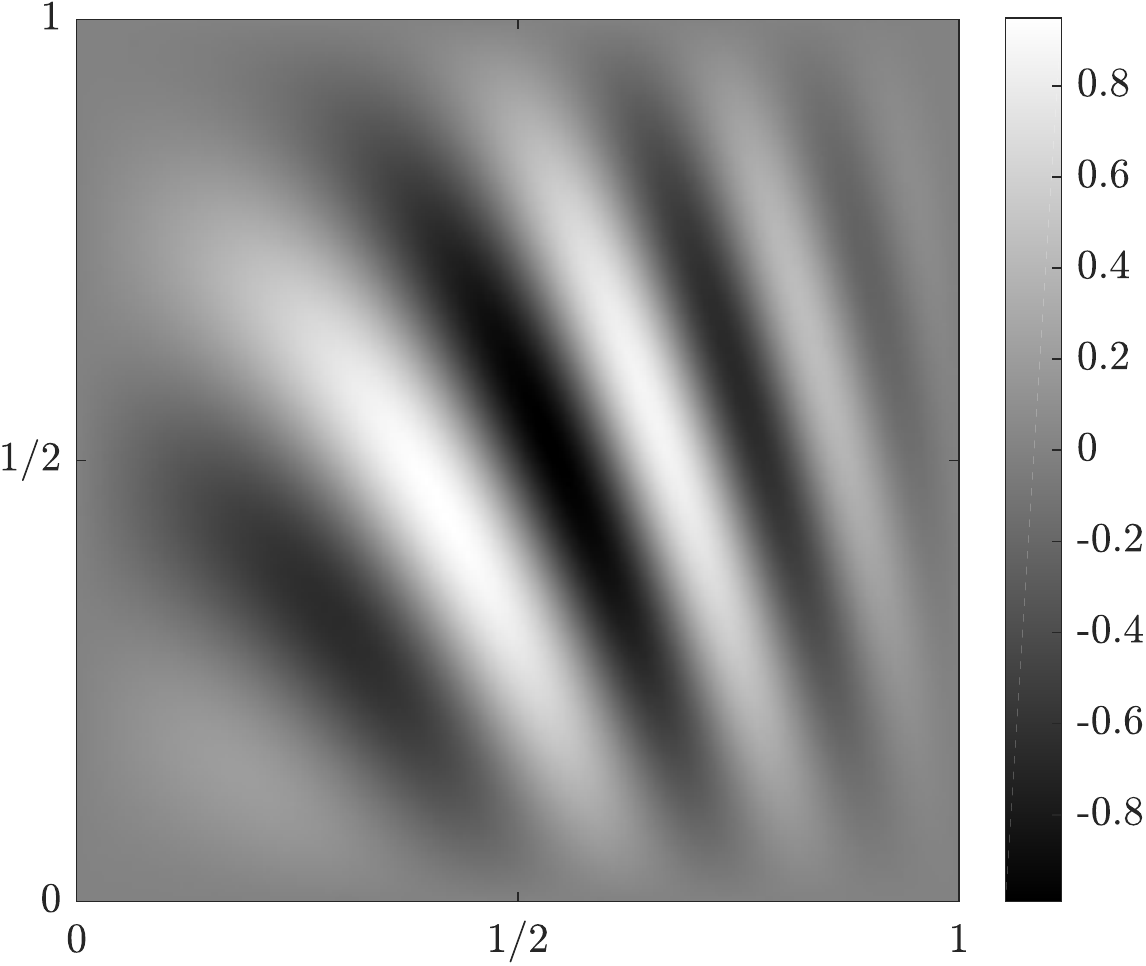} &
\includegraphics[width=.45\textwidth]{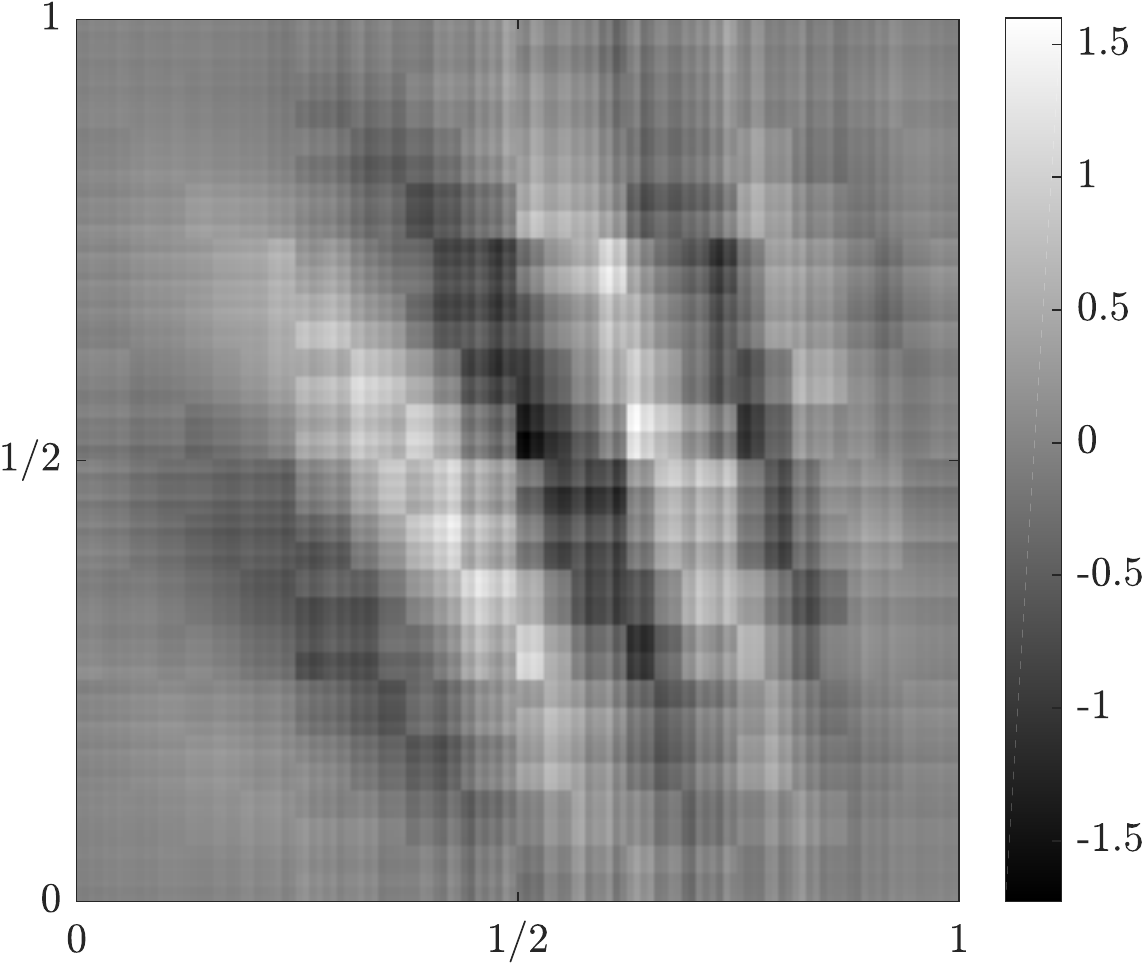} \\
\end{tabular}
\caption{Function $f$ (left) and approximation via Lemma \ref{smolyak} (right).}
\label{fig2}
\end{figure}

Next, we set $c_1 = 8$ and sample $l = c_1 n \log^2 n$ points uniformly at
random from $[0,1]^2$. In Figure \ref{fig3}, we plot the approximation of $f$
via Theorem \ref{thm1}, and the approximation of $f$ via Theorem \ref{thm1} with
$n=128$ spin cycles. In particular, the spin cycles are performed by considering
$f$ as a function on the torus, generating a sequence of random shifts
$\zeta_1,\ldots,\zeta_{n} \in [0,1]^2$, and using the method of Remark
\ref{spin}.

\begin{figure}[h!]
\begin{tabular}{cc}
\includegraphics[width=.45\textwidth]{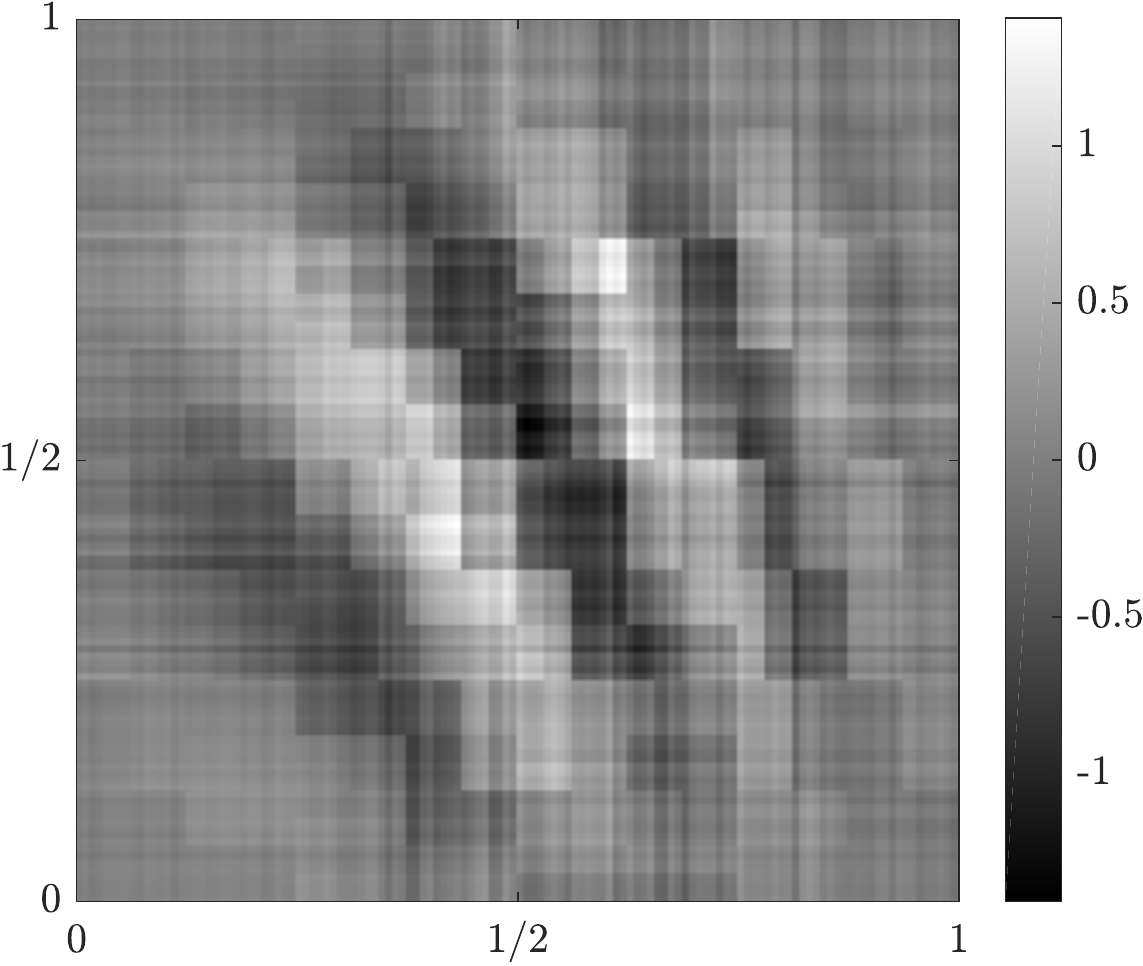} &
\includegraphics[width=.45\textwidth]{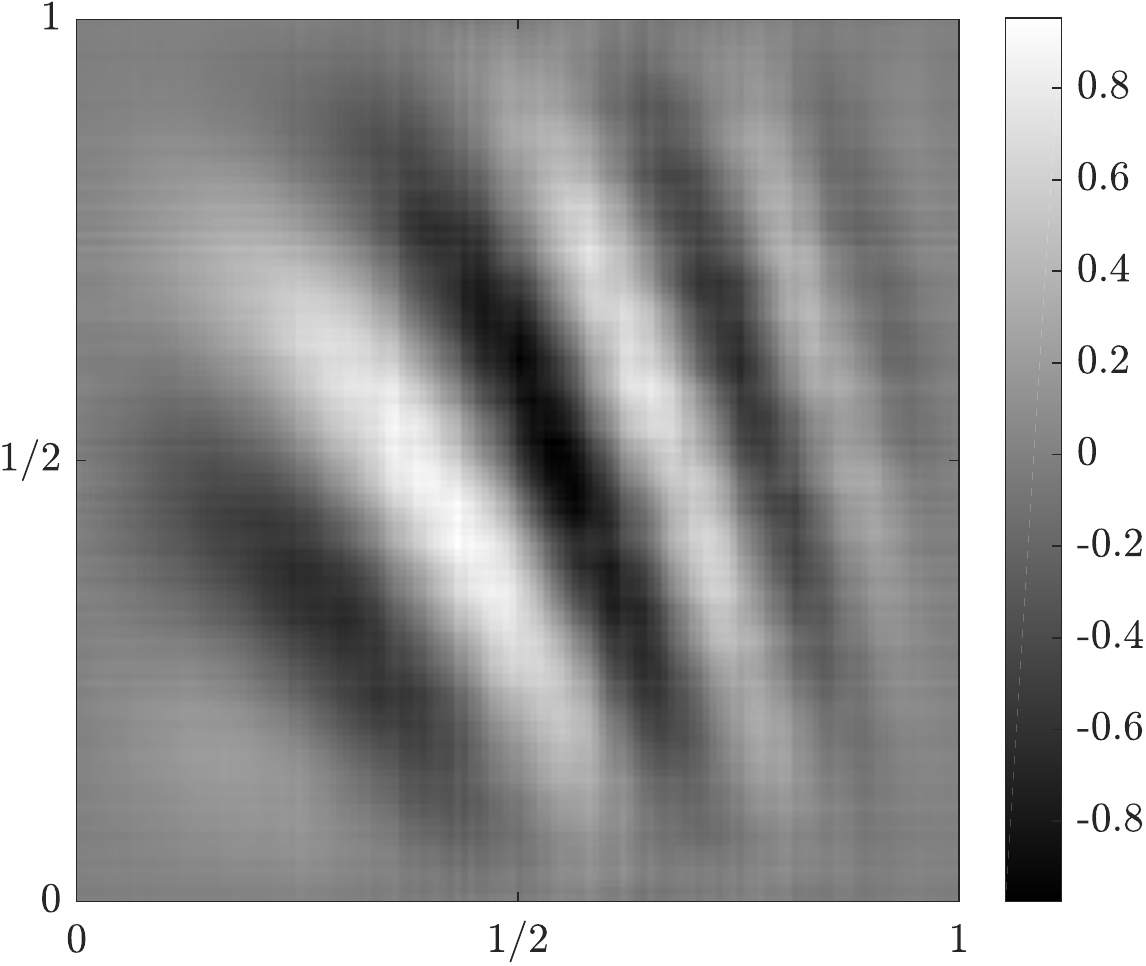}
\end{tabular}
\caption{Approximation via Theorem \ref{thm1} (left) and approximation via
Theorem \ref{thm1} with $128$ spin cycles as in Remark \ref{spin}
(right).} \label{fig3}
\end{figure}

The plots in Figure \ref{fig3} provide empirical evidence that
spin cycling as described in Remark \ref{spin} reduces
artifacts.  Developing quantitative estimates for improvements in approximation
accuracy resulting from spin cycling is an interesting theoretical problem for
future study.

\subsection{Discussion} 
There are several possible extensions and applications of Theorem \ref{thm1}.
Informally speaking, we have shown that in $2$-dimensions the sampling
requirements for the method of Smolyak \cite{Smolyak1963} can be relaxed from a
specific set of points at the center of dyadic rectangles to a similar number of
random samples.  As previously noted, Smolyak \cite{Smolyak1963} presented a
general $d$-dimensional version of Lemma \ref{smolyak} so an immediate question
for future study is the extension of Theorem \ref{thm1} to the $d$-dimensional
cube.  This would require defining a more sophisticated embedding $\Psi$ that
retains an analog of the martingale property established in \S \ref{martingale}.
It may also be interesting to consider generalizations of Theorem \ref{thm1}
to abstract dyadic trees as discussed by M.~Gavish and R.~R.~Coifman in
\cite{GavishCoifman2012}.

There are also interesting theoretical questions in $2$-dimensions related to
random matrix theory. Given a collection of $l$ points $X_1,\ldots,X_l$ chosen
uniformly at random from $[0,1]^2$, we can consider the $l \times (m+2)2^{m-1}$
dimensional matrix $B$ whose $j$-th row is $\Psi(X_j)$, where $\Psi$ is the
embedding defined in Definition \ref{def1}. The rows of $B$ are
independent, and the inner product of a vector with a row of $B$ is a martingale
sum, see \S \ref{martingale}. It would be interesting to develop quantitative
high probability estimates on the singular values of $B$.

Finally, we note that the method of Smolyak \cite{Smolyak1963} has been
developed into a computational method called sparse grids, see
\cite{BungartzGriebel2004}. The relaxation to random sampling and the ability to
perform spin cycles may prove useful for certain applications.  In particular,
it may be interesting to consider applications of Theorem \ref{thm1} in the
Fourier domain,  where the mixed H\"older condition is very natural. Recently, M.~Griebel and J.~Hamaekers \cite{GriebelHamaekers2014} have developed
a fast discrete Fourier transform on sparse grids, which could potentially be
used in combination with Theorem \ref{thm1}.

\subsection*{Acknowledgements}
The author would like to thank Raphy Coifman for many fruitful discussions.

\end{document}